\documentclass[a4paper,UKenglish,numberwithinsect,cleveref, autoref, thm-restate]{lipics-v2021}

\usepackage{graphicx}
\usepackage{amsmath}
\usepackage{amsthm}
\usepackage{float}
\usepackage{amsfonts}
\usepackage{amssymb}

\theoremstyle{definition}
\newtheorem*{notation*}{Notation}

\newcommand{\freccia}[3]{#2 \colon #1  \to #3}

\newcommand{\duefreccia}[3]{\xymatrix@C=0.5cm{#2 \colon #1  \ar@{=>}[r] &  #3}}

\newcommand{\comsquare}[8]{ \xymatrix@+1pc{ 
#1 \ar[r]^{#5} \ar[d]_{#6} & #2 \ar[d]^{#7} \\
#3 \ar[r]_{#8} & #4 
}}
\newcommand{\pullback}[8]{ \xymatrix@+1pc{ 
#1 \pullbackcorner \ar[r]^{#5} \ar[d]_{#6} & #2 \ar[d]^{#7} \\
#3 \ar[r]_{#8} & #4 
}}
\newcommand{\quadratocomm}[8]{ \xymatrix@+1pc{ 
#1 \ar[r]^{#5} \ar[d]_{#6} & #2 \ar[d]^{#7} \\
#3 \ar[r]_{#8} & #4 
}}
\newcommand{\comsquarelargo}[8]{ \xymatrix@+1pc{ 
#1 \ar[rr]^{#5} \ar[d]_{#6} && #2 \ar[d]^{#7} \\
#3 \ar[rr]_{#8} && #4 
}}
\newcommand{\parallelmorphisms}[4]{\xymatrix@+1pc{
#1 \ar @<+4pt>[r]^{#2} \ar @<-4pt>[r]_{#3} & #4
}}
\newcommand{\relation}[4]{\xymatrix@+1pc{
\angbr{#2}{#3}\colon #1 \ar @<+4pt>[r] \ar @<-4pt>[r] & #4
}}
\newcommand{\frecceparalleleopposte}[4]{\xymatrix@+1pc{
#1 \ar@<+4pt>[r]^{#2} \ar@<-4pt>@{<-}[r]_{#3} & #4
}}
\newcommand{\equalizer}[6]{\xymatrix@+1pc{
#1 \ar[r]^{#2} & #3 \ar @<+4pt>[r]^{#4} \ar @<-4pt>[r]_{#5} & #6
}}
\newcommand{\coequalizer}[6]{\xymatrix@+1pc{
 #1 \ar @<+4pt>[r]^{#2} \ar @<-4pt>[r]_{#3} & #4 \ar[r]^{#5} & #6
}}

\newcommand{\subobject}[3]{\xymatrix{
#1 \ar@{>->}[r]^{#2} & #3
}}

\newcommand{\pullbackcorner}[1][ul]{\save*!/#1+1.2pc/#1:(1,-1)@^{|-}\restore}

\def\mC{\mathcal{C}}

\def\mD{\mathcal{D}}

\def\Set{\mathbf{Set}}
\def\Rel{\mathbf{Rel}}

\def\id{\operatorname{ id}}         

\newcommand{\angbr}[2]{\langle #1,#2 \rangle}

\usepackage{mathtools}
\usepackage{quiver}
\usepackage{cleveref}
\usepackage{hyperref}
\usepackage{xcolor}
\usepackage[all,2cell]{xy}
\UseAllTwocells
\xyoption{v2}
\usepackage{tikz}
\usepackage{tikzit}
\nolinenumbers
\usetikzlibrary{shapes}

\usepackage{freetikz}
\usetikzlibrary{decorations.markings,positioning,patterns}
\usetikzlibrary{shadows}
\usepackage{array}

\tikzstyle{nodonero}=[fill=black, draw=black, shape=circle]
\tikzstyle{box}=[fill=white, draw=black, shape=rectangle]
\tikzstyle{medium box}=[fill=white, draw=black, shape=rectangle, minimum width=0.7cm, minimum height=0.7cm]
\tikzstyle{bn}=[fill=black, draw=black, shape=circle, inner sep=1.5pt]
\tikzstyle{state}=[fill=white, draw=black, regular polygon, regular polygon sides=3, minimum width=0.8cm, shape border rotate=180, inner sep=0pt]
\tikzstyle{costate}=[fill=white, draw=black, regular polygon, regular polygon sides=3, minimum width=0.8cm, inner sep=0pt]
\tikzstyle{comp}=[fill={rgb,255: red,191; green,0; blue,64}, draw={rgb,255: red,191; green,0; blue,64}, shape=circle, inner sep=1.5pt]

\tikzstyle{ds}=[-, dashed, dash pattern=on 1mm off 1mm]

\usepackage{todonotes}

\theoremstyle{plain} 
\newtheorem{mytheorem}{Theorem}[section]

\newtheorem{mylemma}[mytheorem]{Lemma}
\newtheorem{myproposition}[mytheorem]{Proposition}

\theoremstyle{definition} 
\newtheorem{mydefinition}[mytheorem]{Definition}

\newtheorem{myremark}[mytheorem]{Remark}
\newtheorem{myexample}[mytheorem]{Example}

\title{Weakly Markov categories\\ and weakly affine monads}
\titlerunning{Weakly Markov categories and weakly affine monads}

\author{Tobias Fritz}{Department of Mathematics, University of Innsbruck, AT}{tobias.fritz@uibk.ac.at}{}{FWF P 35992-N}
\author{Fabio Gadducci}{Department of Computer Science, University of Pisa, Pisa, IT}{fabio.gadducci@unipi.it}{https://orcid.org/
0000-0003-0690-3051}{MIUR PRIN 2017FTXR ``IT-MaTTerS''.}
\author{Paolo Perrone}{Department of Computer Science, University of Oxford, UK}{paolo.perrone@cs.ox.ac.uk}{https://orcid.org/0000-0002-9123-9089}{ERC Grant ``BLaSt -- A Better Language for Statistics''.}
\author{Davide Trotta}{Department of Computer Science, University of Pisa, Pisa, IT}{trottadavide92@gmail.com}{https://orcid.org/0000-0003-4509-594X}{MIUR PRIN 2017FTXR ``IT-MaTTerS''.}
\authorrunning{T. Fritz \emph{et alii}}

\Copyright{Tobias Fritz and Fabio Gadducci and Paolo Perrone and Davide Trotta}


\begin{document}

\maketitle

\begin{abstract}
   Introduced in the 1990s in the context of the algebraic approach to graph rewriting,
   gs-monoidal categories are symmetric monoidal categories
   where each object is equipped with the structure of a commutative comonoid. They arise for example as
   Kleisli categories of commutative monads on cartesian categories,
   and as such they provide a general framework for effectful computation.
   Recently proposed in the context of categorical probability, Markov categories are
   gs-monoidal categories where the monoidal unit is also terminal, and they arise for example as
   Kleisli categories of commutative \emph{affine} monads, where affine means that the monad preserves the monoidal unit.

   The aim of this paper is to study a new condition on the gs-monoidal structure, resulting in the concept of \emph{weakly Markov categories},
   which is intermediate between gs-monoidal categories and Markov ones.
   In a weakly Markov category, the morphisms to the monoidal unit are not necessarily unique, but form a group.
   As we show, these categories exhibit a rich theory of conditional independence for morphisms, generalising the known theory for Markov categories.
   We also introduce the corresponding notion for commutative monads, which we call weakly affine, and for which we give two equivalent
   characterisations.

   The paper argues that these monads are relevant to the study of categorical probability.
   A case at hand is the monad of finite non-zero measures, which is weakly affine but not affine.
   Such structures allow to investigate probability without normalisation within an elegant categorical framework.
\end{abstract}

\section{Introduction}

The idea of \emph{gs-monoidal categories}, which are symmetric monoidal categories equipped with copy and discard morphisms making every object a comonoid,
was first introduced in the context of algebraic approaches to term graph rewriting~\cite{CorradiniGadducci97}, and then
developed in a series of papers \cite{CorradiniGadducci99, CorradiniGadducci99b, CorradiniGadducci02}.
Two decades later, similar structures have been rediscovered independently in the context of categorical probability theory, in particular in \cite{cho_jacobs_2019}
and \cite{Fritz_2020}, under the names of \emph{copy-discard (CD) categories} and \emph{Markov categories}.
While ``CD-categories'' and ``gs-monoidal categories'' are synonyms, Markov categories have the additional condition that the monoidal unit is the terminal object
(i.e.~every morphism commutes with the discard maps), a condition corresponding to normalisation of probability. See~\cite[Remark~2.2]{fritz2022}
for a more detailed history of these ideas.

A canonical way of obtaining a gs-monoidal category is as the Kleisli category of a commutative monad on a cartesian monoidal category.
As argued in \cite{kock2012distributions}, commutative monads can be seen as generalising theories of distributions of some kind, and the fact that their Kleisli categories are gs-monoidal can be seen as the correspondence between distributions and (possibly unnormalised) probability theory.
In particular, when the monad is affine (i.e.~it preserves the monoidal unit \cite{Kock71,Jacobs1994}), the Kleisli category is Markov -- this can be seen as the correspondence between normalised distributions and probability theory.

In this work we introduce and study an intermediate notion between gs-monoidal and Markov categories, which we call \emph{weakly Markov categories}.
These are defined as gs-monoidal categories where for every object its morphisms to the monoidal unit form a group (\Cref{defweaklymarkov}).
Weakly Markov categories can be interpreted intuitively as gs-monoidal categories where each morphism is discardable up to an invertible normalisation (see \Cref{weaklydiscardable} for the precise mathematical statement). The choice of the name is due to the fact that every Markov category is (trivially) weakly Markov.

In parallel to weakly Markov categories, we also introduce \emph{weakly affine monads}, which are commutative monads on cartesian monoidal categories preserving the (internal) group structure of the terminal object (\Cref{defweaklyaffine}).
As a particular concrete example of relevance to probability and measure theory, we consider the monad of finite non-zero measures on $\Set$ (\Cref{ex:nonzero_measures}), and we use it as a running example in the rest of the work.
As we show (see \Cref{weaklyboth}), a commutative monad on a cartesian monoidal category is weakly affine if and only if its Kleisli category is weakly Markov, analogously to what happens with affine monads and Markov categories.

Markov categories come equipped with a notion of \emph{conditional independence}, which has been one of the main motivations for their use in categorical probability and statistics \cite{cho_jacobs_2019,Fritz_2020,fritz2022dseparation}.
It is noteworthy that a notion of conditional independence can also be given for any gs-monoidal category. As we show, for weakly Markov categories it has convenient properties that can be considered ``up-to-normalisation'' versions of their corresponding Markov-categorical counterpart.
These concepts allow us to provide an equivalent condition for weak affinity of a monad, namely a pullback condition on the associativity diagram of the structural morphisms $c_{X,Y}:TX\times TY\to T(X\times Y)$ (\Cref{mainthm}), widely generalising the elementary statement that a monoid is a group if and only if its associativity diagram is a pullback (\Cref{assoc_group}).
As such, we believe that weakly affine monads are relevant to the study of categorical probability,
as they allow to investigate probability without normalisation within an elegant categorical framework.

\paragraph*{Categorical probability}

In the past few years, we have seen a rapid increase in the interest for categorical methods in probability and information theory, and we briefly sketch the basic ideas in order to provide context for this paper.

The first works on categories of stochastic maps, almost as old as category theory and information theory themselves, have been proposed by Lawvere~\cite{lawvere_unpublished} and independently by Chentsov~\cite{chentsov_first}.
Subsequently, similar intuitions have been expressed in terms of monads by Giry~\cite{Giry82}, as well as by \'Swirszcz in the context of convex analysis~\cite{swirszcz}, and by Jones and Plotkin in computer science~\cite{jones-plotkin}.
These monads are collectively and informally known as \emph{probability monads}.

Today, Markov categories \cite{cho_jacobs_2019,Fritz_2020}, which generalise categories of stochastic maps and Kleisli categories of probability monads, have been used to express probabilistic concepts \emph{synthetically}, i.e.~in terms of basic axioms that the categories satisfy, and of which the usual measure-theoretic proofs are a concrete instance.

The advantages of a categorical approach to probability theory are multiple. First of all, it facilitates an almost-verbatim translation of probabilistic ideas into programming languages, in particular probabilistic programming languages, even in the case of highly complex models.
Also, categorical probability comes with a graphical formalism similar to the one of Bayesian networks (see~\cite{fritz2022dseparation} for more details), allowing to represent the structure of stochastic interactions in terms of a graph, for easier interpretation by a human. It is a high-order language, in the sense that it expresses visually some ideas of measure-theoretic significance without requiring measure theory itself, similar to how high-level programming languages spare the programmer from working directly with machine code.
Finally, the categorical formalism \emph{complements} the traditional measure-theoretic one in the sense that several concepts which are hard to express or prove with one method are easier to approach using the other method, once the main structures are in place. In this sense, categorical probability is a novel, additional box of tools which provides shortcuts to proofs that would otherwise be lengthy and counterintuitive.

Among concepts that have been expressed and proven this way, we have
de Finetti's theorem~\cite{fritz2021definetti},
the Kolmogorov extension theorem and the Kolmogorov and Hewitt-Savage zero-one laws~\cite{fritzrischel2019zeroone},
a categorical d-separation criterion~\cite{fritz2022dseparation},
theorems on multinomial and hypergeometric distributions~\cite{jacobs2021multinomial},
theorems on sufficient statistics~\cite{Fritz_2020} and on comparison of statistical experiments~\cite{fritz_2020_B},
data processing inequalities in information theory~\cite{ours_entropy},
the ergodic decomposition theorem in dynamical systems~\cite{moss2022ergodic},
and results on fresh name generation in theoretical computer science~\cite{fritz2022dilations}.

\paragraph*{Outline}
In \Cref{secbackground} we review the main structures used in this work, in particular group and monoid objects, gs-monoidal and Markov categories, and their interaction with commutative monads.

In \Cref{secweakly} we define the main original concepts, namely weakly Markov categories and weakly affine monads. We study their relationship and we prove that a commutative monad on a cartesian monoidal category is weakly affine if and only if its Kleisli category is weakly Markov (\Cref{weaklyboth}).
We then turn to concrete examples using finite measures and group actions (\Cref{secexamples}).

In \Cref{secindep} we extend the concept of conditional independence from Markov categories to general gs-monoidal categories. We specialise to the weakly Markov case and show that the situation is then similar to what happens in Markov categories, but in a certain precise sense only up to normalisation.
We use this formalism to equivalently reformulate weak affinity in terms of a pullback condition (\Cref{mainthm}). Together with the newly introduced concepts,
this result can be considered the main outcome of our work.

Finally, in the concluding \Cref{secfurther}, we pose further questions, such as when we can iterate the construction of weakly Markov categories by means of weakly affine monads, and the relation to strongly affine monads in the sense of Jacobs~\cite{Jacobs16}.

\section{Background}
\label{secbackground}

In this section, we develop some relevant background material for later reference.
To begin, the following categorical characterisation of groups will be useful to keep in mind.
\begin{myproposition}\label{assoc_group}
 A monoid $(M,m,e)$ in $\Set$ is a group if and only if the associativity square
 \begin{equation}\label{assoc}
  \begin{tikzcd}
   M \times M \times M \ar{d}{\id\times m} \ar{r}{m\times\id} & M\times M \ar{d}{m} \\
   M\times M \ar{r}{m} & M
  \end{tikzcd}
 \end{equation}
 is a pullback.
\end{myproposition}
\begin{proof}
 The square~\eqref{assoc} is a pullback of sets if and only if given $a,g,h,c\in M$ such that $ag=hc$, there exists a unique $b\in M$ such that $g=bc$ and $h=ab$.
 First, suppose that $G$ is a group. Then the only possible choice of $b$ is
 \[
  b = a^{-1}h = gc^{-1}
 \]
 which is unique by uniqueness of inverses.

 Conversely, suppose that \eqref{assoc} is a pullback. We can set $g,h=e$ and $c=a$ so that $ae=ea=a$.
 Instantiating the pullback property on these elements gives $b$ such that $ab=e$ and $ba=e$, that is, $b=a^{-1}$.
\end{proof}

\Cref{assoc_group} holds generally for a monoid object in a cartesian monoidal category, where the element-wise proof still applies thanks to the following standard observation.

\begin{myremark}
\label{yonedaremark}
	Given an object $M$ in a cartesian monoidal category $\mD$, there is a bijection between internal monoid structures on $M$ and monoid structures on every hom-set $\mD(X, M)$ such that pre-composition with any $f : X \to Y$ defines a monoid homomorphism
	\[
		\mD(Y, M) \longrightarrow \mD(X, M).
	\]
	The proof is straightforward by the Yoneda lemma.
	It follows that \Cref{assoc_group} holds for internal monoids in cartesian monoidal categories in general.
\end{myremark}

For the consideration of categorical probability, we now recall the simplest version of a commutative monad of measures.
It works with measures taking values in any semiring instead of $[0,\infty)$ (see e.g.~\cite[Section~5.1]{coumans2013scalars}),
but we restrict to the case of $[0,\infty)$ for simplicity.

\begin{mydefinition}\label{monadM}
 Let $X$ be a set. Denote by $MX$ the set of \emph{finitely supported measures on $X$}, i.e.~the functions $m:X\to[0,\infty)$
 that are zero for all but a finite number of $x\in X$.
 Given a function $f:X\to Y$, denote by $Mf:MX\to MY$ the function sending $m\in MX$ to the assignment
 \[
	 (Mf)(m) \: : \: y \longmapsto \sum_{x\in f^{-1}(y)} m(x) .
 \]
 This makes $M$ into a functor, and even a monad with the unit and multiplication maps
 \[
  \begin{tikzcd}[row sep=0]
   X \ar{r}{\delta} & MX \\
   x \ar[mapsto]{r} & \delta_x ,
  \end{tikzcd}
  \qquad\qquad
  \begin{tikzcd}[row sep=0]
   MMX \ar{r}{E} & MX \\
   \xi \ar[mapsto]{r} & E\xi ,
  \end{tikzcd}
 \]
 where
 \[
  \delta_x(x') = \begin{cases}
                  1 & x=x' , \\
                  0 & x\ne x',
                 \end{cases}
 \qquad\qquad
 (E\xi)(x) = \sum_{m\in MX} \xi(m)\,m(x) .
 \]
 Call $M$ the \emph{measure monad} on $\Set$.

 Denote also by $DX\subseteq MX$ the subset of \emph{probability measures}, i.e.~those finitely supported $p:X\to[0,\infty)$ such that
 \[
  \sum_{x\in X} p(x) = 1 .
 \]
 $D$ forms a sub-monad of $M$ called the \emph{distribution monad}.
\end{mydefinition}

It is known that $M$ is a commutative monad~\cite{coumans2013scalars}.
The corresponding lax monoidal structure
\[
	MX \times MY \stackrel{c}{\longrightarrow} M(X \times Y)
\]
is exactly the formation of product measures given by $c(m, m')(x, y) = m(x) m'(y)$.
Also $D$ is a commutative monad with the induced lax monoidal structure,
since the product of probability measures is again a probability measure.

\subsection{GS-monoidal and Markov categories}

We recall here the basic definitions adopting the graphical formalism of string diagrams,
referring to \cite{Selinger2011} for some background on various notions of monoidal categories and their associated diagrammatic calculus.

\begin{mydefinition}
A \textbf{gs-monoidal category} is a symmetric monoidal category $(\mC, \otimes, I)$
with a commutative comonoid structure on each object $X$ consisting of a comultiplication
and a counit
\ctikzfig{copy_del}
which satisfy the commutative comonoid equations
\ctikzfig{comonoid_equation}
These comonoid structures must be multiplicative with respect to the monoidal structure,
meaning that it satisfies the equations
\ctikzfig{comon-struct-mult}
\end{mydefinition}

\begin{mydefinition}
 A morphism $f:X\to Y$ in a gs-monoidal category is called \textbf{copyable} or \textbf{functional} if
 \ctikzfig{functional}
 It is called \textbf{discardable} or \textbf{full} if
 \ctikzfig{full}
\end{mydefinition}

\begin{myexample}
The category $\Rel$ of sets and relations with the monoidal operation $\otimes : \Rel \times \Rel\to \Rel$ given by the direct product of sets is a gs-monoidal category \cite{CorradiniGadducci02}. In this gs-monoidal category, the copyable arrows are precisely the partial functions, and the discardable arrows are the total relations.
\end{myexample}

\begin{myremark}\label{cartesiancase}
It is well-known that if every morphism is copyable and discardable, or equivalently if the copy and discard maps are natural, then the monoidal product is the categorical product, and thus the category is cartesian monoidal \cite{Fox:CACC}.
In other words, the following conditions are equivalent for a gs-monoidal category $\mC$
 \begin{itemize}
  \item $\mC$ is cartesian monoidal;
  \item every morphism is copyable and discardable;
  \item the copy and discard maps are natural.
 \end{itemize}
\end{myremark}

In recent work~\cite{fgtc2022lax} it has been shown that gs-monoidal categories naturally arise in several ways, such as Kleisli categories of commutative monads or span categories. In the following proposition, we recall the result regarding Kleisli categories.

\begin{myproposition}\label{monoidalgs}
 Let $T$ be a commutative monad on a cartesian monoidal category $\mD$.
 Then its Kleisli category $\mathrm{Kl}_T$ is canonically a gs-monoidal category with the copy and discard structure induced by that of $\mD$.
\end{myproposition}

\begin{myexample}\label{kleisliM}
 The Kleisli categories of the monads $M$ and $D$ of \Cref{monadM} are gs-monoidal. We can write their Kleisli categories concretely as follows
 \begin{itemize}
  \item a morphism $k:X\to Y$ of $\mathrm{Kl}_M$ is a \emph{matrix} with rows indexed by $Y$ and columns indexed by $X$, and non-negative entries $k(y|x)$ such that for each $x\in X$, the number $k(y|x)$ is non-zero only for finitely many $y$;
  \item a morphism $k:X\to Y$ of $\mathrm{Kl}_D$ is a morphism of $\mathrm{Kl}_M$ such that moreover, for all $x\in X$, the sum of each column satisfies
  \[
   \sum_{y\in Y} k(y|x) = 1 
  \]
  If $X$ and $Y$ are finite, such a matrix is called a \emph{stochastic matrix}.
 \end{itemize}

 In both categories, identities are identity matrices, composition is matrix composition, monoidal structure is the cartesian product on objects and the Kronecker product on matrices, and the copy and discard maps are the images of the standard copy and discard maps on $\Set$ under the Kleisli inclusion functor.
\end{myexample}

\emph{Markov categories}~\cite{Fritz_2020} represent one of the more interesting specialisations of the notion of gs-monoidal category.
Based on the interpretation of their arrows
as generalised Markov kernels, they are considered the foundation of a categorical approach to probability theory.

\begin{mydefinition}
 A gs-monoidal category is said to be a \textbf{Markov category} if any (hence all) of the following equivalent conditions are satisfied
 \begin{itemize}
  \item the monoidal unit is terminal;
  \item the discard maps are natural;
  \item every morphism is discardable.
 \end{itemize}
 \end{mydefinition}

We recall from \cite{Kock71,Jacobs1994} the notion of \emph{affine monad}.

\begin{mydefinition}
 A monad $T$ on a cartesian monoidal category is called \textbf{affine} if $T1\cong 1$.
\end{mydefinition}

It was observed in~\cite[Corollary~3.2]{Fritz_2020} that if the monad preserves the terminal object, then every arrow of the Kleisli category is discardable, and this makes the Kleisli category into a Markov category.
Since the converse is easy to see, we have the following addendum to \Cref{monoidalgs}.

\begin{myproposition}\label{affinemarkov}
Let $T$ be a commutative monad on a cartesian monoidal category $\mD$. Then $\mathrm{Kl}_T$ is Markov if and only if $T$ is affine.
\end{myproposition}

\begin{myexample}
 The distribution monad $D$ of \Cref{monadM} is affine, and so its Kleisli category (\Cref{kleisliM}) is a Markov category. It is one of the simplest examples of categories of relevance for categorical probability.

The measure monad $M$ is not affine, as it is easy to see that $M1\cong[0,\infty)$, and so its Kleisli category is not Markov.
\end{myexample}

\section{Weakly Markov categories and weakly affine monads}
\label{secweakly}

In this section, we introduce an intermediate level between gs-monoidal and Markov called \emph{weakly Markov}, and its corresponding notion for monads, which we call \emph{weakly affine}.

\subsection{The monoid of effects}\label{monoids}

In a gs-monoidal category $\mC$ we call a \emph{state} a morphism from the monoidal unit $p:I\to X$, and \emph{effect} a morphism to the monoidal unit $a:X\to I$.
As is standard convention, we represent such morphisms as triangles as follows
 \ctikzfig{state-costate}

 Effects, i.e.~elements of the set $\mC(X,I)$, form canonically a commutative monoid as follows: the monoidal unit is the discard map $X\to I$, and given $a,b:X\to I$, their product $ab$ is given by copying\footnote{See also e.g.~the $\odot$ product in~\cite[Proposition~3.10]{coecke2011phasegroups}.}
\ctikzfig{ptwise-product}
If a morphism $f:X\to Y$ is copyable and discardable, the pre-composition with $f$ induces a morphism of monoids $\mC(Y,I)\to\mC(X,I)$.

\begin{myremark}
 The monoidal unit $I$ of a monoidal category is canonically a monoid object via the coherence isomorphisms $I\otimes I\cong I$ and $I\cong I$.
 However, in a general (i.e. not necessarily cartesian) gs-monoidal category $\mC$, the monoid structure on $\mC(X,I)$ is not, as in \Cref{yonedaremark}, coming from considering the presheaf represented by $I$. Indeed, in order for \Cref{yonedaremark} to hold, we would need that every pre-composition is a morphism of monoids. As remarked above, this fails in general unless all morphisms are copyable and discardable (i.e.~if $\mC$ is not cartesian monoidal).
\end{myremark}

Let us now consider the case where the gs-monoidal structure comes from a commutative monad on a cartesian monoidal category $\mD$.
In this case, the monoid structure on Kleisli morphisms $X\to 1$ does come from the canonical internal monoid structure on $T1$ (and from the one on $1$) in $\mD$.
Indeed, $T1$ is a monoid object with the following unit and multiplication~\cite[Section~10]{kock2012distributions}
 \[
 \begin{tikzcd}
  1 \ar{r}{\eta} & T1 ,
 \end{tikzcd}
 \qquad
 \begin{tikzcd}
  T1 \times T1 \ar{r}{c_{1,1}} & T(1\times 1) \ar{r}{\cong} & T1 .
 \end{tikzcd}
 \]
 For example, for the monad of measures $M$, we obtain $M1=[0,\infty)$ with its usual multiplication.
 The resulting monoid structure on Kleisli morphisms $X\to 1$ is now given as follows. The unit is given by
 \[
 \begin{tikzcd}
  X \ar{r}{\mathrm{del}_X} & 1 \ar{r}{\eta} & T1 ,
 \end{tikzcd}
 \]
 and the multiplication of Kleisli morphisms $f, g : X \to 1$ represented by $f^\sharp,g^\sharp:X\to T1$ is the Kleisli morphism represented by
 \[
 \begin{tikzcd}
  X \ar{r}{\mathrm{copy}_X} & X\times X \ar{r}{f^\sharp\times g^\sharp} &
  T1 \times T1 \ar{r}{c_{1,1}} & T(1\times 1) \ar{r}{\cong} & T1 .
 \end{tikzcd}
 \]
 For the monad of measures $M$, Kleisli morphisms $X\to 1$ are represented by functions $X\to [0,\infty)$, and this description shows that their product is the point-wise product.

For a general $\mC$, the commutative monoid $\mC(X,I)$ acts on the set $\mC(X,Y)$: given $a:X\to I$ and $f:X\to Y$, the resulting $a\cdot f$ is given as follows
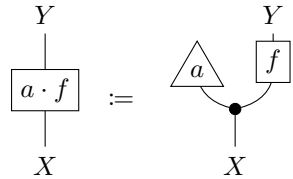
\ctikzfig{action}
It is straightforward to see that this indeed amounts to an action of the monoid $\mC(X,I)$ on the set $\mC(X,Y)$.
For the monad of measures $M$, this action is given by point-wise rescaling.

Moreover, for a general $\mC$ the operation
\begin{align*}
	\mC(X,Y) \times \mC(X,Z) & \longrightarrow \mC(X,Y \otimes Z)		\\
	(f,g) & \longmapsto f\cdot g \coloneqq (f\otimes g)\circ\mathrm{copy}_X
\end{align*}
commutes with this action in each variable (separately).

\subsection{Main definitions}

\begin{mydefinition}\label{defweaklymarkov}
 A gs-monoidal category $\mC$ is called \textbf{weakly Markov} if for every object $X$, the monoid $\mC(X,I)$ is a group.
\end{mydefinition}

Clearly, every Markov category is weakly Markov: for every object $X$, the monoid $\mC(X,I)$ is the trivial group.

\begin{mydefinition}
 Given two parallel morphisms $f,g:X\to Y$ in a weakly Markov category $\mC$, we say that $f$ and $g$ are called \textbf{equivalent}, denoted $f\sim g$, if they lie in the same orbit for the action of $\mC(X,I)$, i.e.~if there is $a\in \mC(X,I)$ such that $a\cdot f=g$.
\end{mydefinition}

Note that if $a\cdot f=g$ for some $a$, then $a$ is unique. This can be seen by discarding $Y$ in the following diagram
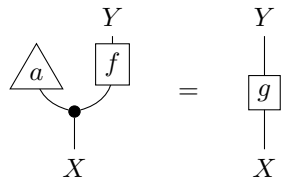
\ctikzfig{free-action}
which shows that taking $a \coloneqq (\mathrm{del_Y} \, f)^{-1} \cdot (\mathrm{del}_Y \, g)$ is the only possibility.
In other words, the action of $\mC(X,I)$ on $\mC(X,Y)$ is free, i.e.~it has trivial stabilisers.

For the next statement, let us first call the \emph{mass} of a morphism $f:X\to Y$ in a gs-monoidal category $\mC$
the morphism $m_f\coloneqq \mathrm{del}_Y\circ f:X\to I$.
Note that $f$ is discardable if and only if $m_f=\mathrm{del}_X$, i.e.~if its mass is the unit of the monoid $\mC(X,I)$.

\begin{myproposition}\label{weaklydiscardable}
 Every morphism $f:X\to Y$ in a weakly Markov category is equivalent to a unique discardable morphism.
\end{myproposition}

We call the discardable morphism the \emph{normalisation} of $f$ and denote it by $n_f:X\to Y$.

\begin{proof}
 Consider the mass $m_f$, and denote its group inverse by $m_f^{-1}$. The morphism $n_f\coloneqq m_f^{-1}\cdot f$ is discardable and equivalent to $f$.
 Suppose now that $d:X\to Y$ is discardable and equivalent to $f$, i.e.~there exists $a:X\to I$ such that $d = a\cdot f$.
 Since $d$ is discardable
 \ctikzfig{normalization_proof}
 which means that $a=m_f^{-1}$, i.e.~$d=n_f$.
\end{proof}

In other words, every morphism $f$ can be written as its mass times its normalisation.

Let us now look at the Kleisli case.

\begin{mydefinition}\label{defweaklyaffine}
 A commutative monad $T$ on a cartesian monoidal category is called \textbf{weakly affine} if $T1$ with its canonical internal commutative monoid structure is a group.
\end{mydefinition}

This choice of terminology is motivated by the following proposition, which can be seen as a ``weakly'' version of \Cref{affinemarkov}.

\begin{myproposition}\label{weaklyboth}
 Let $\mD$ be a cartesian monoidal category and $T$ a commutative monad on $\mD$. Then the Kleisli category of $T$ is weakly Markov if and only if $T$ is weakly affine.
\end{myproposition}

\begin{proof}
 First, suppose that $T1$ is an internal group, and denote by $\iota:T1\to T1$ its inversion map.
 The inverse of a Kleisli morphism $a : X \to 1$ in $\mathrm{Kl}_T(X,1)$ represented by $a^\sharp:X\to T1$ is represented by $\iota\circ a^\sharp$: indeed, the following diagram in $\mD$ commutes
 \[
  \begin{tikzcd}
  X \ar[bend right=8pc]{dd}[swap]{\mathrm{del}_X} \ar{d}[swap]{a^\sharp} \ar{r}{\mathrm{copy}_X} & X\times X \ar{d}[swap]{a^\sharp\times a^\sharp} \ar{dr}{a^\sharp\times(\iota\circ a^\sharp)} \\
  T1 \ar{d}[swap]{\mathrm{del}_{T1}} \ar{r}[swap]{\mathrm{copy}_{T1}} & T1\times T1 \ar{r}[swap]{\id\times\iota} & T1\times T1 \ar{r}{c_{1,1}} & T(1\times 1) \ar{d}{\cong} \\
  1 \ar{rrr}{\eta} &&& T1
  \end{tikzcd}
 \]
 where the bottom rectangle commutes since $\iota$ is the inversion map for $T1$. The analogous diagram with $\iota\times\id$ in place of $\id\times\iota$
 similarly commutes.

 Conversely, suppose that for every $X$, the monoid structure on $\mathrm{Kl}_T(X,1)$ has inverses. Then in particular we can take $X=T1$, and the inverse of the Kleisli morphism $\id:T1\to T1$ is an inversion map for $T1$.
\end{proof}

This result can also be thought of in terms of the Yoneda embedding, via \Cref{yonedaremark}: since the Yoneda embedding preserves and reflects pullbacks (and all limits), the associativity square for $T1$ is a pullback in $\mD$ if and only if the associativity squares of all the monoids $\mD(X,T1)$ are pullbacks.
Note that \Cref{yonedaremark} applies since we are assuming that $\mD$ is \emph{cartesian} monoidal. In the proof of \Cref{weaklyboth}, this is reflected by the fact in the main diagram, the morphism $a^\sharp$ commutes with the copy maps.

\subsection{Examples of weakly affine monads}
\label{secexamples}

Every affine monad is a weakly affine monad. Below you find a few less trivial examples.

\begin{myexample}
    \label{ex:nonzero_measures}
	Let $M^* : \Set \to \Set$ be the monad assigning to every set the set of finitely supported discrete \emph{non-zero}
	measures on $M^*$, or equivalently let $M^*(X)$ for any set $X$ be the set of non-zero finitely supported functions $X \to [0,\infty)$.
	It is a sub-monad $M^* \subseteq M$, meaning that the monad structure is defined in terms of the same formulas as for the monad of measures $M$ (\Cref{monadM}).
	Similarly, the lax structure components
	\[
		c_{X,Y} \: : \: M^* X \times M^* Y \longrightarrow M^*(X \times Y)
	\]
	are also given by the formation of product measures, or equivalently point-wise products of functions
	$X \to [0,\infty)$.

	Since $M^* 1 \cong (0,\infty)\ncong 1$, this monad is not affine. However the monoid structure of $(0,\infty)$ induced by $M^*$ is the usual multiplication of positive real numbers, which form a group. Therefore $M^*$ is weakly affine, and its Kleisli category is weakly Markov.

	On the other hand, if the zero measure is included, we have $M1 \cong [0,\infty)$ which is not a group under multiplication, so $M$ is not weakly affine.
\end{myexample}

\begin{myexample}
	\label{ex:abelian_group}
	Let $A$ be a commutative monoid.
	Then the functor $T_A \coloneqq A \times -$ on $\Set$ has a canonical structure of commutative monad,
	where the lax structure components $c_{X,Y}$ are given by multiplying elements in $A$ while carrying the elements
	of $X$ and $Y$ along.

	Since $T_A 1 \cong A$, the monad $T_A$ is weakly affine if and only if $A$ is a group, and affine if and only if $A\cong 1$.
\end{myexample}

\begin{myexample}
    \label{ex:nonexample}
    As for negative examples,
    consider the free abelian group monad $F$ on $\Set$. Its functor takes a set $X$ and forms the set $FX$ of finite multisets (with repetition, where order does not matter) of elements of $X$ and their formal inverses.
    We have that $F1\cong \mathbb{Z}$, which is an abelian group under addition.
    However, the monoid structure on $F1$ induced by the monoidal structure of the monad corresponds to the \emph{multiplication} on $\mathbb{Z}$, which does not have inverses. Therefore $F$ is not weakly affine.
\end{myexample}

\section{Conditional independence in weakly Markov categories}
\label{secindep}

Markov categories have a rich theory of conditional independence in the sense of probability theory~\cite{fritz2022dseparation}.
It is noteworthy that some of those ideas can be translated and generalised to the setting of weakly Markov categories.

\begin{mydefinition}\label{defcondind}
 A morphism $f:A\to X_1\otimes\dots\otimes X_n$ in a gs-monoidal category $\mC$ is said to exhibit \textbf{conditional independence of the $X_i$ given $A$} if and only if it can be expressed as a product of the following form
 \ctikzfig{cond-ind-sep}
\end{mydefinition}

Note that this formulation is a bit different from the earlier definitions given in \cite[Definition~6.6]{cho_jacobs_2019} and~\cite[Definition~12.12]{Fritz_2020}, which were formulated for morphisms in Markov categories and state that $f$ exhibits conditional independence if the above holds with the $g_i$ being the \emph{marginals} of $f$, which are
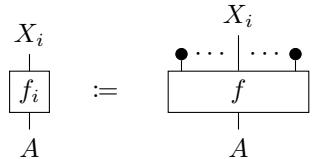
\ctikzfig{marginal}
Indeed, in a Markov category, conditional independence in our sense holds if and only if it holds with $g_i = f_i$~\cite[Lemma~12.11]{Fritz_2020}.
We also say that $f$ is the \emph{product of its marginals}.

\begin{myexample}
 In the Kleisli category of the distribution monad $D$, which is Markov, a morphism $f:A\to X\otimes Y$ exhibits conditional independence if and only if its value at every $a \in A$ is the product of its marginals~\cite[Section~12]{Fritz_2020}.
\end{myexample}

Here is what conditional independence looks like in the Kleisli case.

\begin{myproposition}\label{indepkleisli}
 Let $\mD$ be a cartesian monoidal category and $T$ a commutative monad on $\mD$.
 Then a Kleisli morphism represented by $f^\sharp:A\to T(X_1\times\dots\times X_n)$ exhibits conditional independence of the $X_i$ given $A$ if and only if it factors as
 
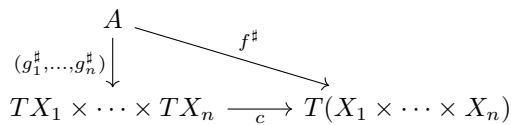
\[
 \begin{tikzcd}
  A \ar{d}[swap]{(g_1^\sharp,\dots,g_n^\sharp)} \ar{dr}{f^\sharp} \\
  TX_1\times\dots\times TX_n \ar{r}[swap]{c} & T(X_1\times\dots\times X_n)
 \end{tikzcd}
 \]
 for some Kleisli maps $g_i^\sharp:A\to TX_i$,
 where the map $c$ above is the one obtained by iterating the lax monoidal structure (which is unique by associativity).
\end{myproposition}
\begin{proof}
 In terms of the base category $\mD$, a Kleisli morphism in the form of \Cref{defcondind} reads as follows
 \[
  \begin{tikzcd}[sep=large]
   A \ar{r}{\mathrm{copy}} & A\times\dots\times A \ar{r}{g_1^\sharp\times\dots\times g_n^\sharp} & TX_1\times\dots\times TX_n \ar{r}{c} & T(X_1\times\dots\times X_n) .
  \end{tikzcd}
 \]
 Therefore $f^\sharp:A\to T(X_1\times\dots\times X_n)$ exhibits the conditional independence if and only if it is of the form above.
\end{proof}

\begin{myexample}\label{zeromeasure}
	In the Kleisli category of the measure monad $M$, and for any object, the morphism $A \to X_1 \otimes \cdots \otimes X_n$ given by the zero measure on every $a \in A$ exhibits conditional independence of its outputs given its input. For example, for $A=1$, the zero measure on $X\times Y$ is the product of the zero measure on $X$ and the zero (or any other) measure on $Y$.
 Notice that both marginals of the zero measure are zero measures -- therefore, the factors appearing in the product are not necessarily related to the marginals.
\end{myexample}

In a weakly Markov category, the situation is similar to the Markov case discussed above,
but up to equivalence: an arrow exhibits conditional independence if and only if it is \emph{equivalent to} the product of its marginals.

\begin{myproposition}\label{eqcondind}
 Let $f:A\to X_1\otimes\dots\otimes X_n$ be a morphism in a weakly Markov category $\mC$. Then $f$ exhibits conditional independence of the $X_i$ given $A$ if and only if it is equivalent to the product of all its marginals.
\end{myproposition}

\begin{proof}
 Denote the marginals of $f$ by $f_1,\dots,f_n$.
 Suppose that $f$ is a product as in \Cref{defcondind}. By marginalising, for each $i=1,\dots,n$ we get
  \ctikzfig{cond-ind-proof}
 Therefore for each $i$ we have that $f_i\sim g_i$.

 Conversely, suppose that $f$ is equivalent to the product of its marginals, i.e.~that there exists $a:X\to I$ such that $f$ is equal to the following
 \ctikzfig{cond-ind-proof2}
 One can then choose $g_i=f_i$ for all $i<n$, and $g_n = a\cdot f_n$, so that $f$ is in the form of \Cref{defcondind}.
\end{proof}

\begin{myremark}
 For $n=2$, a morphism $f:A\to X\otimes Y$ in a weakly Markov category $\mC$ exhibits conditional independence of $X$ and $Y$ given $A$
 if and only if the equation below holds
 \ctikzfig{cond-ind2}
 Indeed this arises as a consequence of \cref{eqcondind} by noting that both sides of the equation describe the same element of $\mC(A,I)$ upon marginalising.
\end{myremark}

\subsection{Main result}

The concept of conditional independence for weakly Markov categories allows us to give an equivalent characterisation of weakly affine monads.
The condition is a pullback condition on the associativity diagram, and it recovers \Cref{assoc_group} when applied to the monads of the form $A \times -$ for $A$ a commutative monoid.

\begin{mytheorem}\label{mainthm}
 Let $\mD$ be a cartesian monoidal category and $T$ a commutative monad on $\mD$.
 Then the following conditions are equivalent
 \begin{enumerate}
  \item\label{condgroup} $T$ is weakly affine;
  \item\label{condwm} the Kleisli category $\mathrm{Kl}_T$ is weakly Markov;
  \item\label{condpullback} for all objects $X$, $Y$, and $Z$, the following associativity diagram is a pullback
	\begin{equation}
		\label{assoc_pullback}
		\hspace{-9pt}	
		\begin{tikzcd}[column sep=3.3pc]
			T(X) \times T(Y) \times T(Z) \ar{r}{\id \times c_{Y,Z}} \ar[swap]{d}{c_{X,Y} \times \id}	& T(X) \times T(Y \times Z) \ar{d}{c_{X,Y \times Z}}	\\
			T(X \times Y) \times T(Z) \ar{r}{c_{X\times Y,Z}}						& T(X \times Y \times Z)
		\end{tikzcd}
	\end{equation}
 \end{enumerate}
\end{mytheorem}

In order to prove the theorem above, we will exploit the following property of weakly Markov categories.

\begin{mylemma}[localised independence property]\label{local}
 Let $\mC$ be a weakly Markov category. Whenever a morphism  $f:A\to X\otimes Y\otimes Z$ exhibits conditional independence of $X\otimes Y$ (jointly) and $Z$ given $A$, as well as conditional independence of $X$ and $Y\otimes Z$ given $A$, then it exhibits conditional independence of $X$, $Y$, and $Z$ given $A$.
\end{mylemma}

\begin{proof}[Proof of \Cref{local}]
 Let us then assume that $f:A\to X\otimes Y\otimes Z$ exhibits conditional independence of $X\otimes Y$ (jointly) and $Z$ given $A$, as well as conditional independence of $X$ and $Y\otimes Z$ given $A$.
 By marginalising out $X$, we have that $f_{YZ}$ exhibits conditional independence of $Y$ and $Z$ given $A$.
 Since by hypothesis $f$ exhibits conditional independence of $X$ and $Y\otimes Z$ given $A$, by \Cref{eqcondind} it follows that $f$ is equivalent to the product  of $f_X$ and $f_{YZ}$. But, again by \Cref{eqcondind}, $f_{YZ}$ is equivalent to the product of $f_Y$ and $f_Z$, so it follows that $f$ is equivalent to the product of all its marginals. Using \Cref{eqcondind} in the other direction, this means that $f$ exhibits conditional independence of $X$, $Y$, and $Z$ given $A$.
\end{proof}

We are now ready to prove the theorem.

\begin{proof}[Proof of \Cref{mainthm}]
We already know that
 $\ref{condgroup}\Leftrightarrow\ref{condwm}$:
 see \Cref{weaklyboth}.
 We then focus on the correspondence between the first and third item.

 $\ref{condgroup}\Rightarrow\ref{condpullback}$:
 By the universal property of products, a cone over the cospan in \eqref{assoc_pullback} consists of maps $g_1^\sharp:A\to TX$, $g_{23}^\sharp:A\to T(Y\times Z)$, $g_{12}^\sharp:A\to T(X\times Y)$ and $g_3^\sharp:A\to TZ$ such that the following diagram commutes
 \[
  \begin{tikzcd}[column sep=3.3pc]
   A \ar[bend left=10]{drr}{(g_1^\sharp,g_{23}^\sharp)} \ar[bend right=15]{ddr}[swap]{(g_{12}^\sharp,g_3^\sharp)} \\
   & T(X) \times T(Y) \times T(Z) \ar{r}[swap]{\id \times c_{Y,Z}} \ar{d}{c_{X,Y} \times \id}	& T(X) \times T(Y \times Z) \ar{d}{c_{X,Y \times Z}}	\\
   & T(X \times Y) \times T(Z) \ar{r}[swap]{c_{X\times Y,Z}}						& T(X \times Y \times Z)
  \end{tikzcd}
 \]

\noindent
By \Cref{indepkleisli}, this amounts to a Kleisli morphism $f^\sharp:A\to T(X\times Y\times Z)$ exhibiting conditional independence of $X$ and $Y\otimes Z$ given $A$, as well as of $X\otimes Y$ and $Z$ given $A$. By the localised independence property (\Cref{local}), we then have that $f$ exhibits conditional independence of all $X$, $Y$ and $Z$ given $A$, and so, again by \Cref{indepkleisli}, $f^\sharp$ factors through the product $TX\times TY\times TZ$.
 More specifically, by marginalising over $Z$, we have that $g_{12}^\sharp$ factors through $TX\times TY$, i.e.~the following diagram on the left commutes for some $h_1^\sharp:A\to TX$ and $h_2^\sharp:A\to TY$, and similarly, by marginalising over $X$, the diagram on the right commutes for some $\ell_2^\sharp:A\to TY$ and $\ell_3^\sharp:A\to TZ$
 \[
  \begin{tikzcd}
   A \ar{d}[swap]{(h_1^\sharp,h_2^\sharp)} \ar{dr}{g_{12}^\sharp} \\
   TX\times TY \ar{r}[swap]{c_{X,Y}} & T(X\times Y)
  \end{tikzcd}
  \qquad
  \begin{tikzcd}
   A \ar{d}[swap]{(\ell_2^\sharp,\ell_3^\sharp)} \ar{dr}{g_{23}^\sharp} \\
   TY\times TZ \ar{r}[swap]{c_{Y,Z}} & T(Y\times Z)
  \end{tikzcd}
 \]

 \noindent
 In other words, we have that the upper and the left curved triangles in the following diagram commute
\[
  \begin{tikzcd}[column sep=3.3pc]
   A \ar[bend left=15, shift left]{drr}{(g_1^\sharp,g_{23}^\sharp)} \ar[bend right=35]{ddr}[swap]{(g_{12}^\sharp,g_3^\sharp)}
    \ar[shift left]{dr}[near end]{(g_1^\sharp,\ell_2^\sharp,\ell_3^\sharp)} \ar[shift right]{dr}[swap, pos=0.7]{(h_1^\sharp,h_2^\sharp, g_3^\sharp)} \\
   & T(X) \times T(Y) \times T(Z) \ar{r}[swap]{\id \times c_{Y,Z}} \ar{d}{c_{X,Y} \times \id}	& T(X) \times T(Y \times Z) \ar{d}{c_{X,Y \times Z}}	\\
   & T(X \times Y) \times T(Z) \ar{r}[swap]{c_{X\times Y,Z}}						& T(X \times Y \times Z)
  \end{tikzcd}
 \]

 \noindent
 By marginalising over $Y$ and $Z$, and by weak affinity of $T$, there exists a unique $a^\sharp:A\to T1$ such that $h_1 = a\cdot g_1$.
 Therefore
 \[
  g_{12} = h_1\cdot h_2 = (a\cdot g_1) \cdot h_2 = g_1\cdot (a\cdot h_2) ,
 \]

 \noindent
 and so in the diagram above we can equivalently replace $h_1$ and $h_2$ with $g_1$ and $a\cdot h_2$.
 Similarly, by marginalising over $X$ and $Y$, there exists a unique $c^\sharp:A\to T1$ such that $\ell_3=c\cdot g_3$, so that
 \[
  g_{23}= \ell_2\cdot\ell_3 = \ell_2\cdot (c\cdot g_3) = (c\cdot \ell_2) \cdot g_3
 \]

 \noindent
 and in the diagram above we can replace $\ell_2$ and $\ell_3$ with $c\cdot \ell_2$ and $g_3$, as follows
 \[
  \begin{tikzcd}[column sep=3.3pc]
   A \ar[bend left=15, shift left]{drrr}{(g_1^\sharp,g_{23}^\sharp)} \ar[bend right=35]{ddrr}[swap]{(g_{12}^\sharp,g_3^\sharp)}
    \ar[shift left]{drr}[near end]{(g_1^\sharp,(c\cdot\ell_2)^\sharp,g_3^\sharp)} \ar[shift right]{drr}[swap, pos=0.8]{(g_1^\sharp,(a\cdot h_2)^\sharp, g_3^\sharp)} \\
   && T(X) \times T(Y) \times T(Z) \ar{r}[swap]{\id \times c_{Y,Z}} \ar{d}{c_{X,Y} \times \id}	& T(X) \times T(Y \times Z) \ar{d}{c_{X,Y \times Z}}	\\
   && T(X \times Y) \times T(Z) \ar{r}[swap]{c_{X\times Y,Z}}						& T(X \times Y \times Z)
  \end{tikzcd}
 \]

 \noindent
 Now, marginalising over $X$ and $Z$, we see that necessarily $a\cdot h_2=c\cdot \ell_2$.
 Therefore there is a unique map $A\to TX\times TY\times TZ$ making the whole diagram commute, which means that \eqref{assoc_pullback} is a pullback.

 $\ref{condpullback}\Rightarrow\ref{condgroup}$:
 If $T$ is weakly affine, then taking $X = Y = Z = 1$ in~\eqref{assoc_pullback} shows that this monoid must be an abelian group: we obtain a unique arrow $\freccia{T1}{\iota}{T1}$ making the following diagram commute
	\[
		\begin{tikzcd}
			T1 \ar{dr}[description]{(\id,\iota,\id)} \ar[bend left=15]{drrr}{(\id,\eta_1 \mathrm{del}_{T1})} \ar[bend right,swap]{dddr}{(\eta_1 \mathrm{del}_{T1},\id)} \\
			&	T1 \times T1 \times T1 \ar{r}{\id \times c_{1,1}} \ar[swap]{d}{c_{1,1} \times \id}	& T1 \times T(1\times 1) \ar{d}{c_{1,1\times 1}} \ar{r}{\cong} & T1\times T1 \ar{d}{c_{1,1}}	\\
			&	T(1\times 1) \times T1 \ar{r}{c_{1\times 1,1}} \ar{d}[swap]{\cong}	& T(1\times 1\times 1) \ar{d}{\cong} \ar{r}{\cong} & T(1\times 1) \ar{d}{\cong} \\
			& T1\times T1 \ar{r}[swap]{c_{1,1}} & T(1\times 1) \ar{r}[swap]{\cong} & T1
		\end{tikzcd}
	\]
	and the commutativity shows that $\iota$ satisfies the equations making it the inversion map for a group structure.
\end{proof}

\begin{myexample}
 In the Kleisli category of the measure monad $\mathrm{Kl}_M$ (which is not weakly affine) consider the following diagram
 \[
		\hspace{-9pt}	
		\begin{tikzcd}[column sep=3.3pc]
			MX \times MY \times MZ \ar{r}{\id \times c_{Y,Z}} \ar[swap]{d}{c_{X,Y} \times \id}	& MX \times M(Y \times Z) \ar{d}{c_{X,Y \times Z}}	\\
			M(X \times Y) \times MZ \ar{r}{c_{X\times Y,Z}}						& M(X \times Y \times Z)
		\end{tikzcd}
	\]
	In the top-right corner $MX\times M(Y\times Z)$, take the pair $(0,p)$ where $p$ is any non-zero measure on $Y\times Z$, and similarly, in the bottom-left corner take the pair $(q,0)$ where $q$ is any non-zero measure on $X\times Y$. Following the diagram, both pairs are mapped to the zero measure in the bottom-right corner. If the diagram was a pullback, we would be able to express the top-right and bottom-left corners as coming from the same triple in $MX\times MY\times MZ$, that is, there would exist a measure $m$ on $Y$ such that $m\cdot 0=p$ and $0\cdot m=q$. Since $p$ and $q$ are non-zero, this is not possible.
\end{myexample}
\begin{myremark}
It is worth noting that the pullback condition on the associativity square is not equivalent to the localised independence property of \Cref{local}: recall that a zero measure always exhibits conditional independence of all its outputs (\Cref{zeromeasure}).
Therefore, for zero measures, the localised independence property is always trivially valid, and hence the Kleisli category of the measures monad $M$ satisfies it in general.
However, the example above shows explicitly that the pullback property fails.

For now it is an open question whether the localised independence property for a Kleisli category is reflected by an equivalent condition on the monad.
\end{myremark}

\section{Conclusions and future work}
\label{secfurther}

Our paper introduces weakly Markov categories and weakly affine monads and explore their relationship.
More explicitly, our main result (\Cref{mainthm}) establishes a tight correspondence between the algebraic properties
of $T1$ and the universal properties of certain commutative squares given by the structural arrows of $T$
for a commutative monad $T$ on a cartesian category.
We believe that this theorem suggests at least two directions 
\begin{itemize}
\item generalising the statement to weakly affine monads on weakly Markov categories;
\item generalising other Markov-categorical notions, such as the positivity axiom, to weakly Markov or even gs-monoidal categories.
\end{itemize}
We will provide further details on these potential directions in what follows.
\smallskip

\textbf{Regarding possible generalisations.} In \Cref{mainthm}, we provide a characterisation of weakly affine monads on cartesian monoidal categories. Taking inspiration from the case of affine monads on Markov categories~\cite[Corollary~3.2]{Fritz_2020}, it seems natural to consider whether our main result can be extended to commutative monads on \emph{weakly Markov categories}.

However, this problem is non-trivial and requires clever adjustments to the main definitions. The crucial point is that, in general, the structure of the internal group of $T1$ and the structure of the group $\mD(X,T1)$ are not necessarily related in the current definitions. One approach could be to introduce a form of \emph{compatibility} for $T1$ and $\mD(X,T1)$ by defining a weakly affine monad on a weakly Markov category as a commutative monad such that $T1$ is an internal group and $\mD(X,T1)$ is a group with the composition and units induced by those of $T1$. With this change, for example, \Cref{weaklyboth} would work for any weakly Markov category, but \Cref{mainthm} would likely fail as its proof involves the universal property of products.
\smallskip

\textbf{On the positivity axiom.}
A strong monad $T$ on a cartesian monoidal category is \emph{strongly affine}~\cite{Jacobs16} if for every pair of objects $X$ and $Y$ the following diagram is a pullback
 \[
  \begin{tikzcd}
   X \times TY \ar{d}{\pi_1} \ar{r}{s} & T(X\times Y) \ar{d}{T\pi_1} \\
   X \ar{r}{\eta} & TX
  \end{tikzcd}
 \]
where $s$ denotes the strength and $\eta$ denotes the unit of the monad. Every strongly affine monad is affine.
The corresponding condition on the Markov category $\mathrm{Kl}_T$ has recently been characterised as an information flow axiom called \emph{positivity}~\cite[Section~2]{fritz2022dilations}.

For a generic commutative monad, the diagram above may even fail to commute (take the measure monad $M$ and start with $(x,0)$ in the top left corner). One can however consider the following diagram, which reduces to the one above (up to isomorphism) in the affine case
\[
 \begin{tikzcd}
   X \times TY \ar{d}{\id\times T(\mathrm{del}_{Y})} \ar{r}{s} & T(X\times Y) \ar{d}{T(\id\times \mathrm{del}_Y)} \\
   X \times T1 \ar{r}{s} & T(X\times 1) \cong TX
  \end{tikzcd}
\]
and which always commutes by naturality of the strength.
One can then call the monad $T$ \emph{positive} if this second diagram is a pullback. Upon defining \emph{positive gs-monoidal categories} analogously to positive Markov categories, one may conjecture that $T$ is positive if and only if $\mathrm{Kl}_T$ is positive.
This would generalise the existing result for Markov categories.

\bibliographystyle{plain}

\bibliography{biblio_davide}

\end{document}